\documentclass[a4paper,twoside,12pt]{article}
\usepackage{amssymb,amsmath,amsthm,latexsym}
\usepackage{amsfonts}
\usepackage{graphicx,caption,subcaption}
\usepackage[pdftex,bookmarks,colorlinks=false]{hyperref}
\usepackage[hmargin=1.2in,vmargin=1.2in]{geometry}
\usepackage[mathscr]{euscript}

\usepackage[auth-lg,affil-sl]{authblk}
\usepackage{float}

\def\N{\mathbb{N}}
\def\C{\mathbb{C}}

\newcommand{\J}{\mathscr{J}}

\newtheorem{theorem}{Theorem}[section]
\newtheorem{definition}[theorem]{Definition}

\newtheorem{corollary}[theorem]{Corollary}

\title{\textbf{\sc On $J$-Colouring of Chithra Graphs}}

\author{Johan Kok$^1$, Sudev Naduvath$^2$\footnote{Corresponding Author}}
\affil{\small Centre for Studies in Discrete Mathematics\\ Vidya Academy of Science \& Technology \\ Thalakkottukara, Thrissur-680501, Kerala, India.\\{\tt $^1$kokkiek2@tshwane.gov.za},{\tt $^2$sudevnk@gmail.com}}

\date{}

\begin{document}
\maketitle

\begin{abstract}
\noindent The family of Chithra graphs is a wide ranging family of graphs which includes any graph of size at least one. Chithra graphs serve as a graph theoretical model for genetic engineering techniques or for modelling natural mutation within various biological networks found in living systems. In this paper, we discuss recently introduced $J$-colouring of the family of Chithra graphs. 
\end{abstract}


\vspace{0.25cm}

\noindent\textbf{AMS Classification Numbers:} 05C15, 05C38, 05C75, 05C85.

\section{Introduction}

For general notations and concepts in graphs and digraphs see \cite{BM1,FH,DBW}. Unless mentioned otherwise, all graphs $G$ mentioned in this paper are simple and finite graphs. Note that the order and size of a graph $G$ are denoted by $\nu(G)=n$ and $\varepsilon(G)=p$. The minimum and maximum degrees of $G$ are respectively denoted bby $\delta(G)$ and $\Delta(G)$. The degree of a vertex $v \in V(G)$ is denoted $d_G(v)$ or simply by $d(v)$, when the context is clear. 

We recall that if $\mathcal{C}= \{c_1,c_2,c_3,\dots,c_\ell\}$ and $\ell$ sufficiently large, is a set of distinct colours, a proper vertex colouring of a graph $G$ denoted $\varphi:V(G) \mapsto \mathcal{C}$ is a vertex colouring such that no two distinct adjacent vertices have the same colour. The cardinality of a minimum set of colours which allows a proper vertex colouring of $G$ is called the chromatic number of $G$ and is denoted $\chi(G)$. When a vertex colouring is considered with colours of minimum subscripts the colouring is called a \textit{minimum parameter} colouring. Unless stated otherwise we consider minimum parameter colour sets throughout this paper. The number of times a colour $c_i$ is allocated to vertices of a graph $G$ is denoted by $\theta(c_i)$ and $\varphi:v_i \mapsto c_j$ is abbreviated, $c(v_i) = c_j$. Furthermore, if $c(v_i) = c_j$ then $\iota(v_i) = j$. We shall also colour a graph in accordance with the Rainbow Neighbourhood Convention \cite{KS1}.

\subsection{Rainbow Neighbourhood Convention} 

The rainbow neighbourhood convention has been introduced in \cite{KS1} as follows: Consider a proper minimal colouring $\mathcal{C} =\{c_1,c_2,c_3,\dots,c_\ell\}$ of a graph $G$ under consideration, where $\ell=\chi(G)$. Colour maximum possible number of vertices of $G$ with the colour $c_1$, then colour the maximum possible number of remaining uncoloured vertices with colour $c_2$ and proceeding like this, at the final stage colour the remaining uncoloured vertices by the colour $c_\ell$. Such a colouring may be called a \textit{$\chi^-$-colouring} of a graph.

The inverse to the convention requires the mapping $c_j\mapsto c_{\ell -(j-1)}$. Corresponding to the inverse colouring we define $\iota'(v_i) = \ell-(j-1)$ if $c(v_i) = c_j$.  The inverse of a $\chi^-$-colouring is called a \textit{$\chi^+$-colouring} of $G$.

\subsection{Chithra Graphs}

The family of \textit{Chithra graphs} has been defined in \cite{KS2} as follows: Consider the set $\mathcal{V}(G)$ of all subsets of the vertex set, $V(G)$. Let $\mathcal{W}$ be a collection of any finite number, say $k$, subsets $W_i$, with repetition of selection allowed. That is, $\mathcal{W}= \{W_i: 1 \leq i \leq k, W_i \in \mathcal{V}(G)\}$, with repetition allowed. (This means that, contrary to conventional set theory, we have $\{W_i,\ldots,W_i\} \neq \{W_i\}$). For all non-empty subsets $W_i,\ 1 \leq \ell \leq k$, the corresponding additional vertices $u_1, u_2, u_3, \ldots , u_\ell$, add the edges $u_iv_j$, $\forall v_j \in W_j \subseteq V(G)$. This new graph is called a \textit{Chithra graph} of the given graph $G$ and is  denoted by $\C_{\mathcal{W}}(G)$. The family of Chithra graphs of the graph $G$ is denoted by $\mathfrak{C}(G)$. 

Note in general that the empty subset of $V(G)$ may be selected. However, the empty subset does not represent an additional vertex. Hence, the vertex $u_\emptyset$ is empty. This argument implies that for a set of empty sets, $\emptyset = \bigcup\limits_{i=1}^{k}\emptyset_i$, we have $\C_\emptyset(G)=G \in \mathfrak{C}(G)$ as well. 

Also, note that the Chithra graphs can be constructed from non-connected (disjoint) graphs and from edgeless (null) graphs. It means that for the edgeless graph on $n$ vertices, denoted by $\mathfrak{N}_n$, the complete bi-partite graph, $K_{n,m}\in \mathfrak{C}(\mathfrak{N}_n)$. Equally so, $K_{n,m}\in \mathfrak{C}(\mathfrak{N}_m)$. 

We note that the family of Chithra graphs is a wide ranging family of graphs which includes any graph of size at least one. Compared to split graphs the important relaxation is that a complete component (clique) is not a necessary requirement. A split graph is isomorphic to a Chithra graph on condition that each vertex in the independent set is mapped onto a subset of the set of vertices of the clique component. We say that a split graph is an explicit Chithra graph of a complete graph. A number of well-known classes of graphs are indeed explicit Chithra graphs of some graph $G$. Small graphs such as sun graphs, sunlet graphs, crown graphs and helm graphs are all explicit Chithra graphs of some cycle $C_n$.   Chithra graphs can serve as a graph theoretical model for genetic engineering techniques or for modelling natural mutation within various biological networks found in living systems.

\section{$J$-Colouring of Chithra graphs}

The closed neighbourhood $N[v]$ of a vertex $v \in V(G)$ which contains at least one coloured vertex of each colour in the chromatic colouring, is called a rainbow neighbourhood. We say that vertex $v$ yields a rainbow neighbourhood. We also say that vertex $u \in N[v]$ belongs to the rainbow neighbourhood $N[v]$.

\begin{definition}{\rm \cite{NKS}
A maximal proper colouring of a graph $G$ is a \textit{Johan colouring} or a \textit{$J$-colouring} if and only if every vertex of $G$ belongs to a rainbow neighbourhood of $G$. The maximum number of colours in a $J$-colouring is denoted by $\J(G)$.
}\end{definition}

\begin{definition}{\rm \cite{NKS}
A maximal proper colouring of a graph $G$ is a \textit{modified Johan colouring} or a \textit{$J^*$-colouring} if and only if every internal vertex of $G$ belongs to a rainbow neighbourhood of $G$. The maximum number of colours in a $J^*$-colouring is denoted by $\J^*(G)$.
}\end{definition}

It follows that $\J(\mathfrak{N}_n) = \J^*(\mathfrak{N}_n) = 1,\ n \in \N$ and for any graph $G$ which admits a $J$-colouring, $\chi(G) \leq \J(G)$. Furthermore, if a graph $G$ admits a $J$-colouring it admits a $J^*$-colouring. However, the converse is not always true. 

Unless mentioned otherwise all subsets $W_i \subseteq V(G)$ will be non-empty.

\begin{theorem}\label{Thm-2.1}
If a graph $G$ admits a $J$-colouring and if for each $W_i \in \mathcal{W}$, $W_i =N(v)$ for some $v \in V(G)$, then $\J(\C_{\mathcal{W}}(G))= \J(G)$.
\end{theorem}
\begin{proof}
It is obvious that if $G$ does not admit a $J$-colouring then $\C_{\mathcal{W}}(G)$ cannot admit a $J$-colouring for all $\mathcal{W}\subseteq \mathcal{V}(G)$. Hence, assume that $G$ admits a $J$-colouring. Consider any $W_i = N(v)$ for some $v \in V(G)$. From the definition of the corresponding Chithra graph, $v$ and $u_i$ are non-adjacent and hence $c(u_i)=c(v)$ and this colouring remains as a maximal proper colouring. Also, $N[u_i]$ is a rainbow neighbourhood of $u_i$. By mathematical induction, the same argument follows for all $W_i \in \mathcal{W}$. Therefore, $\C_{\mathcal{W}}(G)$ admits a $J$-colouring. 
\end{proof}

\begin{corollary}\label{Cor-2.2}
If for a graph $G$ each $W_i \in \mathcal{W}$, $W_i =N(v)$ for some $v \in V(G)$, then $\chi(\C_{\mathcal{W}}(G))= \chi(G)$.
\end{corollary}
\begin{proof}
The result follows by similar reasoning in the proof of Theorem \ref{Thm-2.1}.
\end{proof}

\begin{theorem}\label{Thm-2.3}
If a graph $G$ admits a $J$-colouring then, $\J(\C_{\mathcal{W}}(G))= \J(G)+1$  if for each $v \in V(G)$, then there is at least one $W_i \in \mathcal{W}$ such that $N[v] \subseteq W_i $.
\end{theorem}
\begin{proof}
It is obvious that if $G$ does not admit a $J$-colouring, then $\C_{\mathcal{W}}(G)$ cannot admit a $J$-colouring for all $\mathcal{W}\subseteq \mathcal{V}(G)$. Hence, assume that $G$ admits a $J$-colouring. Consider any $W_i$ for which $N[v] \subseteq W_i$ for some $v \in V(G)$. From the definition of the corresponding Chithra graph, $u_i$ is  adjacent to at least $N[v]$ and hence $c(u_i) \neq c(v'),\ v' \in N[v]$. Therefore, $u_i$ requires a new colour say, $c_t$ to ensure a maximal proper colouring. Also, $N[u_i]$ is a rainbow neighbourhood of $u_i$. Since $c(u_i) = c(u_j)=c_t$, $\forall i,j$, all vertices in $V(G)$ yield the new rainbow neighbourhood in $\C_{\mathcal{W}}(G)$. Therefore, $\C_{\mathcal{W}}(G)$ admits a $J$-colouring and $\J(\C_{\mathcal{W}}(G))= \J(G)+1$.
\end{proof}

\begin{corollary}\label{Cor-2.4}
For a graph $G$, $\chi(\C_{\mathcal{W}}(G))= \chi(G)+1$ if for each $v \in V(G)$, then there exists at least one $W_i \in \mathcal{W}$ such that $N[v] \subseteq W_i $.
\end{corollary}
\begin{proof}
The result follows by similar reasoning in the proof of Theorem \ref{Thm-2.3}.
\end{proof}

\begin{corollary}\label{Cor-2.5}
If all $W_i \in \mathcal{W}$, is such that for all $J$-colourings of $G$,  at least one coloured vertex of each colour is an element of $W_i$ and each $v \in V(G)$ is an element of some $W_i$ then, $\J(\C_{\mathcal{W}}(G))= \J(G)+1$. 
\end{corollary}
\begin{proof}
The result follows by similar reasoning in the proof of Theorem \ref{Thm-2.3}.
\end{proof}

\subsection{Special Class of Chithra Graphs}

For $X \subseteq V(G)$, recall that $\langle X\rangle$ denotes the subgraph induced by $X$. Consider a graph $G$ of even order $n$ with the properties that $V(G)$ can be partitioned into $X=\{v_i:1\leq i\leq \frac{n}{2}\}$ and $Y=\{u_i: 1\leq i \leq \frac{n}{2}\}$ such that $N(u_i) = N(v_i),\ \forall\, i$. Then, $\C_Y(G)=\mu(\langle X\rangle)$ is the Mycielskian of $\langle X\rangle$. Then, we have

\begin{theorem}
Whether or not a graph $G$ admits a $J$-colouring, its Mycielskian graph $\mu(G)$ cannot.
\end{theorem}
\begin{proof}
Consider a graph $G$ of order $n$ which admits a proper colouring which satisfies some general condition denoted, $\chi_{gen}(G)$. Let $V(G) = \{v_i:1\leq i \leq n\}$.  Construct the graph $G'$ by adding $n$ additional vertices $Y =\{u_i:1\leq i \leq i\}$ and edges $u_iv_j$ such that $N(u_i) = N(v_i)$. Clearly, $\C_Y(G') = \mu(G)$. By colouring $c(u_i)=c(v_i)$ the graph $G'$ admits a proper colouring which satisfies the same general condition and $\chi_{gen}(G')=\chi_{gen}(G)=|c(N[v_i])|=|c(N[u_i])|$.

We recall that $\chi(\mu(G))=\chi(G)+1$. The aforesaid can be generalise and we can then state that for any graph $G$, we have $\chi_{gen}(\mu(G))=\chi_{gen}(G)+1$.  It implies that the vertex say $x$ which corresponds to $Y$ in $\C_Y(G')$ has $|c(N[x])|=\chi_{gen}(G)+1$. Therefore, irrespective of whether a graph $G$ admits a $J$-colouring or not, its Mycielskian graph $\mu(G)$ does not admit a $J$-coloring.
\end{proof}

\section{Conclusion}

Chithra graphs are extremely general in that any graph with at least one edge is indeed a Chithra graph. To clarify the aforesaid statement, note that if $d_G(v) \geq 1$, then $G'=G-v$ implies that $\C_{N(v)}(G') \simeq G$. Complication in studying Chithra graph is the fact that a graph $G$ can generally be described isomorphic to different Chithra graphs. Admittedly, the description for some classes of graphs up to isomorphism, is unique. For example, for $K_n$, $n \geq 3$ it follows that $K_n = \C_{\{v_1,v_2,v_3,\dots,v_{n-1}\}}(K_{n-1})$. Identifying the appropriate description of a Chithra graph as an explicit Chithra graph of some known graph class to settle particular results is a hard problem.

Corollary \ref{Cor-2.2} and Corollary \ref{Cor-2.4} signal the ease with which most graph invariants can be defined in terms of explicit Chithra graphs. For example, for a minimum dominating set $D$ of a graph $G$ and $W_i \subseteq D$, for each $i$ we have that, $\gamma(\C_{\mathcal{W}}(G)) = \gamma(G)$. All these remarks show there is a wide scope for further research.

\end{document}